\documentclass[12pt]{article}
\usepackage{amsmath,amssymb,amsthm}
\usepackage{fullpage}

\title{An Integral Inequality and  the Riccati--Bernoulli Differential Equation}

\author{Mark B. Villarino\\
Escuela de Matem\'atica, Universidad de Costa Rica,\\
2060 San Jos\'e, Costa Rica}

\date{5 December 2010}

\newtheorem{thm}{Theorem}

\newtheorem*{prob}{Problem}   

\numberwithin{equation}{section}

\makeatletter
\def\section{\@startsection{section}{1}{\z@}{-3.5ex plus -1ex minus
			  -.2ex}{2.3ex plus .2ex}{\large\bf}}
\def\subsection{\@startsection{subsection}{2}{\z@}{-3.25ex plus -1ex
			  minus -.2ex}{1.5ex plus .2ex}{\normalsize\bf}}
\renewcommand{\@dotsep}{200} 
\makeatother

\renewcommand{\leq}{\leqslant}  

\newcommand{\half}{{\mathchoice{\thalf}{\thalf}{\shalf}{\shalf}}}
\newcommand{\shalf}{{\scriptstyle\frac{1}{2}}} 
\newcommand{\thalf}{\tfrac{1}{2}} 
\newcommand{\quarter}{{\mathchoice{\tfourth}{\tfourth}{\sfourth}{\sfourth}}}
\newcommand{\sfourth}{{\scriptstyle\frac{1}{4}}} 
\newcommand{\tfourth}{\tfrac{1}{4}} 

\begin{document}

\maketitle

\begin{abstract}
We apply an integral inequality to obtain a rigorous \textit{a priori}
estimate of the accuracy of the partial sum to the power series
solution of the \textsc{Riccati}--\textsc{Bernoulli} differential
equation.
\end{abstract}


\section{Introduction}

The deservedly celebrated \textsc{Riccati}--\textsc{Bernoulli}
differential equation:
\begin{equation}
\label{RB}
\boxed{y' = x^2 + y^2}
\end{equation}
has the general solution~\cite{EP}:
\begin{equation}
\label{GS}
\boxed{y(x) = 
x \,\frac{J_{\frac{3}{4}}(\half x^2) - c J_{-\frac{3}{4}}(\half x^2)}
{c J_{\frac{1}{4}}(\half x^2) + J_{-\frac{1}{4}}(\half x^2)}}
\end{equation}
where $c$ is an arbitrary constant and where
\begin{equation*}
\boxed{J_n(x) := \frac{x^n}{2^n\Gamma(n+1)} \biggl\{
1 - \frac{x^2}{2^2\cdot 1!\cdot(n+1)}
+ \frac{x^4}{2^4\cdot 2!\cdot(n+1)(n+2)} - \cdots \biggr\}}
\end{equation*}
is the \emph{Bessel function of the first kind} of order~$n$, where
$n$ is any real number, and $\Gamma(n + 1)$ is the \emph{gamma
function}. The very interesting history of~\eqref{RB} is detailed in
\textsc{Watson}'s standard treatise~\cite{WA}.

The equation \eqref{RB} is an example of a simple differential
equation whose solutions form a family of transcendental functions
which are essentially distinct from the elementary transcendents.

Unfortunately, the general solution \eqref{GS} does not easily lend
itself to a rigorous error analysis of its accuracy in a particular
interval of the independent variable.

We will show how a simple application of an \emph{integral inequality}
allows one to estimate the accuracy of the partial sum of the
\textsc{Taylor} series expansion of the solution within the latter's
interval of convergence.


\section{Cauchy's theorem}

(In this section we follow \cite{IPM}, Section~IV.5). The general
\emph{Cauchy Problem} is to solve the ordinary differential equation
initial-value problem:
\begin{equation}
\label{IVP}
\boxed{y' = f(x,y),  \quad  y(x_0) := y_0} \,.
\end{equation}
Let $f(x,y)$ be expanded in the series
\begin{equation}
\label{PS}
f(x,y) = \sum_{i,j} A_{ij} (x - x_0)^i (y - y_0)^j
\end{equation}
convergent for
\begin{equation}
\label{R}
|x - x_0| < R_1, \quad |y - y_0| < R_2, \qquad (R_1 > 0, \ \ R_2 > 0).  
\end{equation}
Then, according to \textsc{Cauchy}'s theorem in the theory of
differential equations, the problem \eqref{IVP} has a solution $y(x)$
represented by the series
\begin{equation}
\label{y(x)}
\boxed{y(x) = \sum_{k=0}^\infty \frac{y^{(k)}(x_0)}{k!} (x - x_0)^k}
\end{equation}
convergent in some neighborhood of the point~$x_0$.

\textsc{Cauchy}'s theorem allows one also to indicate the neighborhood
of the point~$x_0$, in which the series \eqref{y(x)} converges.
Namely, let $M$ be a constant such that
\begin{equation*}
|f(x,y)| \leq M
\end{equation*}
with
\begin{equation*}
|x - x_0| \leq r_1 < R_1,  \qquad  |y - y_0| \leq r_2 < R_2,
\end{equation*}
where $R_1$ and $R_2$ are numbers, defining the region \eqref{R} of
the convergence of the series \eqref{PS}, and $r_1$ and $r_2$ are some
positive numbers. \emph{Then the series \eqref{y(x)} converges for}
\begin{equation*}
|x - x_0| < r,
\end{equation*}
\emph{where}
\begin{equation}
\label{r}
\boxed{r := r_1(1 - e^{-r_2/2Mr_1})} \,.
\end{equation}
It should be noted that the \emph{true} interval of convergence is
usually much larger than~\eqref{r}.


\section{An integral inequality}
We use the notation of \eqref{IVP}.

\begin{thm}
Let $I$ be the interval $x_0 \leq x \leq x_1$, and suppose that
$f(x,y) > 0$ for all $x \in I$, and that the differential inequality
\begin{equation}
\label{DE}
y' \leq f(x_1, y(x))
\end{equation}
also holds there. Then, the integral inequality
\begin{equation}
\label{IE}
\boxed{\int_{x_0}^x
\biggl\{ \frac{1}{f(x_1, y(t))} \,\frac{dy}{dt} \biggr\} \,dt
\leq x - x_0}
\end{equation}
holds for all $x \in I$.
\end{thm}

\begin{proof}
This is a simple consequence of the elementary calculus sufficient
condition that a function be decreasing in an interval. Define
\begin{equation*}
g(x) := \int_{x_0}^x
\biggl\{ \frac{1}{f(x_1, y(t))} \,\frac{dy}{dt} \biggr\} \,dt
- (x - x_0).
\end{equation*}
Taking the derivative and using the fundamental theorem of calculus,
we obtain
\begin{equation*}
g'(x) = \biggl\{ \frac{1}{f(x_1, y(x))} \,\frac{dy}{dx} \biggr\} - 1.
\end{equation*}
But, the inequality \eqref{DE} shows that $g'(x) \leq 0$ for all
$x \in I$.
\end{proof}


\section{The Riccati--Bernoulli Initial Value Problem}
(This section is suggested by \cite{IPM}, Section~IV.5, but we add
some important refinements.)

\begin{prob}
It is required to find the first $11$ terms of the power series
expansion of the solution of the initial value problem
\begin{equation}
\label{IVP2}
\boxed{y' = x^2 + \frac{y^2}{4}, \quad y(0) := -1}
\end{equation}
and an interval of convergence for it.
\end{prob}

If we specialize the general solution \eqref{GS} to this initial value
problem we find the following formula for the exact solution:
\begin{equation}
\label{GS1}
\boxed{y(x) = \frac{x}{16} \cdot 
\frac{\Gamma(\frac{3}{4}) J_{\frac{3}{4}}(\half x^2)
- \sqrt{2}\,\Gamma(\frac{1}{4}) J_{-\frac{3}{4}}(\half x^2)}
{\sqrt{2}\,\Gamma(\frac{1}{4}) J_{\frac{1}{4}}(\half x^2)
+ \Gamma(\frac{3}{4}) J_{-\frac{1}{4}}(\half x^2)} } \,.
\end{equation}
Unfortunately, the computation of the power series solution on the
basis of the quotient \eqref{GS1}, although theoretically possible, is
computationally formidable.

It is easier to use the equation \eqref{IVP2} to compute the
derivatives of $y(x)$ at $x = 0$ directly:
\begin{align}
y' &= x^2 + \frac{y^2}{4} = \frac{1}{4} \,,
\label{y1} \\[\jot]
y'' &= 2 x + \frac{1}{2} yy' = -\frac{1}{8} \,,
\label{y2} \\[\jot]
y''' &= 2 + \frac{1}{2} y'^2 + \frac{1}{2} yy'' = \frac{67}{32} \,,
\label{y3} \\[\jot]
y^{(4)} &= \frac{3}{2} y'y'' + \frac{1}{2} yy''' = -\frac{35}{32} \,,
\label{y4} \\[\jot]
y^{(5)} &= \frac{3}{2} y''^2 + 2 y'y''' + \frac{1}{2} yy^{(4)}
= \frac{207}{128} \,,
\label{y5} \\[\jot]
y^{(6)} &= 5 y''y''' + \frac{5}{2} y'y^{(4)} + \frac{1}{2} yy^{(5)}
= -\frac{231}{64} \,,
\label{y6} \\[\jot]
y^{(7)} &= 5 y'''^2 + \frac{15}{2} y''y^{(4)} + 3 y'y^{(5)}
+ \frac{1}{2} yy^{(6)} = \frac{26585}{1024} \,,
\label{y7} \\[\jot]
y^{(8)} &= \frac{35}{2} y'''y^{(4)} + \frac{21}{2} y''y^{(5)}
+ \frac{7}{2} y'y^{(6)} + \frac{1}{2} yy^{(7)}
= -\frac{119475}{2048} \,,
\label{y8} \\[\jot]
y^{(9)} &= \frac{35}{2} (y^{(4)})^2 + 28 y'''y^{(5)} + 14 y''y^{(6)}
+ 4 y'y^{(7)} + \frac{1}{2} yy^{(8)} = \frac{725769}{4096} \,,
\label{y9} \\[\jot]
y^{(10)} &= 63 y^{(4)}y^{(5)} + 42 y'''y^{(6)} + 18 y''y^{(7)}
+ \frac{9}{2} y'y^{(8)} + \frac{1}{2} yy^{(9)}
= -\frac{10509885}{16384} \,.
\label{y10}
\end{align}
Therefore, by the formula \eqref{y(x)}, the first $11$ terms of the
series solution are:
\begin{align}
y(x) &= -1 + \frac{1}{4} x - \frac{1}{16} x^2 + \frac{67}{192 }x^3
- \frac{35}{768} x^4 + \frac{69}{5120} x^5 - \frac{77}{15360} x^6
+ \frac{5317}{1032192} x^7
\nonumber \\[\jot]
&\qquad - \frac{2655}{1835008} x^8 + \frac{80641}{165150720} x^9
- \frac{77851}{62914560} x^{10} +\cdots
\label{sol}
\end{align}

To find an interval of convergence of this series~\eqref{sol}, we use
the \textsc{Cauchy} theorem. If $x$ and $y$ satisfy the inequalities
\begin{equation*}
|x| \leq 0.5,  \qquad  |y + 1| \leq 1,
\end{equation*}
then we may conclude that
\begin{align*}
|f(x,y)| &\leq \bigl| 0.25[(y + 1) - 1]^2 + x^2 \bigr|
\\
&\leq 0.25(|y + 1| + 1)^2 + |x|^2
\\
&\leq 1.25 \,.
\end{align*}
Therefore, in the formula~\eqref{r}, we may take
\begin{equation*}
r_1 := 0.5,  \qquad  r_2 := 1,  \qquad  M := 1.25,
\end{equation*}
and the value of $r$ we obtain is:
\begin{equation*}
r = 0.5(1 - e^{-0.8}) = 0.2753355\dots
\end{equation*}
Therefore, \emph{the power series solution \eqref{sol} most certainly
converges for $|x| \leq 0.27$}.


\section{The Accuracy of a Partial Sum from the Integral Inequality}

We will consider the following concrete problem, although the
principles are of general applicability.

\begin{prob}
It is required to determine the accuracy of the partial sum of
degree~$9$ of the power series solution \eqref{sol} in the interval
$0 \leq x \leq 0.2$.
\end{prob}

Since the series \eqref{sol} is the \textsc{Maclaurin} expansion of
$y(x)$, we must estimate the remainder term $R_9(x)$, which we write
in the \textsc{Lagrange} form:
\begin{equation*}
\boxed{R_9(x) := \frac{y^{(10)}(\Theta_9)}{10!}\, x^{10}}
\end{equation*}
where $0 < \Theta_9 < 0.2$. We have to estimate $y^{(10)}(x)$, the
formula for which is given in \eqref{y10}, for all values of~$x$ in
the interval $0 \leq x \leq 0.2$. The formulas \eqref{y1} through
\eqref{y9} show us, finally, that \emph{we must estimate $y(x)$ itself
in $0 \leq x \leq 0.2$}.

The estimate of $y(x)$ via \emph{our integral inequality} \eqref{IE}
constitutes the novelty in this paper.

Maintaining the notation of \eqref{IE}, we see that
\begin{equation*}
x_0 := 0,  \qquad   x_1 := 0.2,
\end{equation*}
that the right hand side
\begin{equation*}
f(x,y) := x^2 + \frac{y^2}{4} > 0 
\end{equation*}
on $I$, 
and the differential inequality \eqref{DE} becomes:
\begin{equation*}
\frac{dy}{dx} \leq 0.04 + \frac{y(x)^2}{4}
\end{equation*}
Therefore, the integral inequality \eqref{IE} becomes
\begin{equation*}
\int_0^x \biggl\{ \frac{dy/dt}{0.04 + \quarter y(t)^2} \biggr\} \,dt
\leq x.
\end{equation*}
But
\begin{equation*}
\frac{dy/dt}{0.04 + \quarter y(t)^2}
= \frac{d}{dt} \biggl\{ 10 \arctan \frac{y(t)}{0.4} \biggr\}.
\end{equation*}
Therefore,
\begin{equation*}
10 \arctan \frac{y(x)}{0.4} - 10 \arctan \frac{(-1)}{0.4} \leq x,
\end{equation*}
or
\begin{equation*}
\arctan \frac{y(t)}{0.4} \leq \frac{x}{10} + \arctan \frac{(-1)}{0.4},
\end{equation*}
and taking the tangent of both sides and reducing, \emph{we obtain the
estimate}
\begin{equation}
\label{y11}
\boxed{y(x) 
\leq \frac{\frac{2}{5} \tan(\frac{x}{10}) - 1}{1 + \frac{x}{4}}}
\end{equation}
\emph{which holds for all $x \in I$}.

The function on the right-hand side of \eqref{y11} is monotonically
increasing in $I$,
and
\begin{equation*}
\frac{\frac{2}{5} \tan \bigl( \frac{0.2}{10} \bigr) - 1}
{1 + \frac{0.2}{4}} = -0.9447608\dots < - 0.94,
\end{equation*}
and \emph{we have proved that the following inequality is true for all
$x \in I$}:
\begin{equation}
\label{Y1}
\boxed{-1 \leq y(x) \leq -0.94} \,.
\end{equation}
(\emph{Note:} the true value of $y(0.2)$ is
\begin{equation*}
y(0.2) = -0.9497771\dots,
\end{equation*}
so the estimate \eqref{y11}, with an error of $-0.00501\dots$, or
about $0.53\%$, is quite good!)

Now we must estimate $y'(x),\dots, y^{(10)}(x)$ using the formulas
\eqref{y1} through \eqref{y10}.


\begin{thm}
The following inequalities are valid for all $x\in I$:
\begin{alignat*}{2}
-1 &\leq y(x) &&\leq -0.94, \\
0.22 &\leq y'(x) &&\leq 0.29, \\
-0.15 &\leq y''(x) &&\leq 0.3, \\
1.87 &\leq y'''(x) &&\leq 2.12 \\
-1.13 &\leq y^{(4)}(x) &&\leq -0.74, \\
1.17 &\leq y^{(5)}(x) &&\leq 1.93, \\
-3.38 &\leq y^{(6)}(x) &&\leq 2.23, \\
14.59 &\leq y^{(7)}(x) &&\leq 27.12, \\
-61.96 &\leq y^{(8)}(x) &&\leq -22.73, \\
92.03 &\leq y^{(9)}(x) &&\leq 146.76, \\
\end{alignat*}
and finally
\begin{equation*}
\boxed{-665.9 < y^{(10)}(x) < 281} \,.
\end{equation*}
\end{thm}

\begin{proof}
All of our estimates come from worst-case values applied to each of
the summands in the formulas.

By \eqref{y1},
\begin{equation*}
y' = x^2 + \frac{y^2}{4} \,.
\end{equation*}
Therefore, using \eqref{Y1}, we conclude that for all $x \in I$,
\begin{align*}
y' &<  0.2^2 + \frac{(-1)^2}{4} = 0.29,
\intertext{while}
y' &>  0^2 + \frac{(-.94)^2}{4} = 0.2209 > 0.22.
\end{align*}
Therefore, we obtain the bounds
\begin{equation}
\label{YY1}
\boxed{0.22 < y'(x) < 0.29} \,.
\end{equation}

By \eqref{y2},
\begin{equation*}
y''  =  2x + \frac{1}{2} yy' \,.
\end{equation*}
Therefore, using \eqref{YY1}, we conclude that for all $x\in I$, 
\begin{align*}
y'' &<  2(0.02) + \frac{1}{2}(-0.94)(0.22) = 0.2966 < 0.3,
\intertext{while}
y'' &>  2(0) + \frac{1}{2}(-1)(0.29) = -0.145 >  -0.15.
\end{align*}
Therefore, we obtain the bounds
\begin{equation}
\label{YY2}
\boxed{-0.15 < y''(x) < 0.3} \,.
\end{equation}

By \eqref{y3},
\begin{equation*}
y''' = 2+\frac{1}{2}y'^2+\frac{1}{2}yy'' \,.
\end{equation*}
Therefore, using \eqref{YY1} and \eqref{YY2}, we
conclude that for all $x \in I$,
\begin{align*}
y''' &< 2 + \frac{1}{2}(0.29)^2 + \frac{1}{2}(-0.15)(-1) = 2.11705
< 2.12,
\intertext{while}
y''' &> 2 + \frac{1}{2}(0.22)^2 + \frac{1}{2}(0.3)(-1) = 1.8742 > 1.87.
\end{align*}
Therefore, we obtain the bounds
\begin{equation}
\label{YY3}
\boxed{1.87 < y'''(x) < 2.12} \,.
\end{equation}

By \eqref{y4},
\begin{equation*}
y^{(4)} = \frac{3}{2} y'y'' + \frac{1}{2}yy''' \,.
\end{equation*}
Therefore, using \eqref{YY1}, \eqref{YY2} and \eqref{YY3},
we conclude that for all $x \in I$,
\begin{align*}
y^{(4)} &< \frac{3}{2}(0.29)(0.3) + \frac{1}{2}(-0.94)(1.87) = -0.7484
< -.074,
\intertext{while}
y^{(4)} &> \frac{3}{2}(0.29)(-0.15) + \frac{1}{2}(-1)(2.12) = -1.12525
> -1.13.
\end{align*}
Therefore, we obtain the bounds
\begin{equation}
\label{YY4}
\boxed{-1.13 < y^{(4)}(x) < -0.74} \,.
\end{equation}

By \eqref{y5},
\begin{equation*}
y^{(5)} = \frac{3}{2} y''^2 + 2y'y''' + \frac{1}{2} yy^{(4)} \,.
\end{equation*}
Therefore, using \eqref{YY1} through \eqref{YY4}, we conclude that for
all $x \in I$,
\begin{align*}
y^{(5)} &< \frac{3}{2}(0.3)^2 + 2(0.29)(2.12) + \frac{1}{2}(-1)(-1.13)
= 1.9296 < 1.93,
\intertext{while}
y^{(5)} &> \frac{3}{2}(0)^2 + 2(0.22)(1.87) + \frac{1}{2}(-0.94)(-0.74)
= 1.1706 > 1.17.
\end{align*}
Therefore, we obtain the bounds
\begin{equation}
\label{YY5}
\boxed{1.17 < y^{(5)}(x) < 1.93} \,.
\end{equation}

By \eqref{y6},
\begin{equation*}
y^{(6)} = 5y''y''' + \frac{5}{2} y'y^{(4)} + \frac{1}{2} yy^{(5)} \,.
\end{equation*}
Therefore, using \eqref{YY1} through \eqref{YY5}, we conclude that for
all $x \in I$,
\begin{align*}
y^{(6)} 
&< 5(0.3)(2.12) + \frac{5}{2}(0.22)(-0.74) + \frac{1}{2}(-0.94)(1.17)
= 2.2231 < 2.23,
\intertext{while}
y^{(6)}
&> 5(-0.15)(2.12) + \frac{5}{2}(0.29)(-1.13) + \frac{1}{2}(-1)(1.93)
= -3.37425 > -3.38.
\end{align*}
Therefore, we obtain the bounds
\begin{equation}
\label{YY6}
\boxed{-3.38 < y^{(6)}(x) < 2.23} \,.
\end{equation}

By \eqref{y7},
\begin{equation*}
y^{(7)} = 5y'''^2 + \frac{15}{2} y''y^{(4)} + 3y'y^{(5)}
+ \frac{1}{2} yy^{(6)} \,.
\end{equation*}
Therefore, using \eqref{YY1} through \eqref{YY6}, we conclude that for
all $x \in I$,
\begin{align*}
y^{(7)} &< 5(2.12)^2 + \frac{15}{2}(-0.15)(-1.13) + 3(0.29)(1.93)
+ \frac{1}{2}(-1)(-3.38)
\\
&= 27.11235 < 27.12,
\intertext{while}
y^{(7)} &> 5(1.87)^2 + \frac{15}{2}(0.3)(-1.13) + 3(0.22)(1.17)
+ \frac{1}{2}(-1)(2.23)
\\
&= 14.5992 > 14.59.
\end{align*}
Therefore, we obtain the bounds
\begin{equation}
\label{YY7}
\boxed{14.59 < y^{(7)}(x) < 27.12} \,.
\end{equation}

By \eqref{y8},
\begin{equation*}
y^{(8)} = \frac{35}{2} y'''y^{(4)} + \frac{21}{2} y''y^{(5)}
+ \frac{7}{2} y'y^{(6)} + \frac{1}{2} yy^{(7)} \,.
\end{equation*}
Therefore, using \eqref{YY1} through \eqref{YY7}, we conclude that for
all $x \in I$,
\begin{align*}
y^{(8)} &< \frac{35}{2}(1.87)(-0.74) + \frac{21}{2}(0.3)(1.93) 
+ \frac{7}{2}(0.29)(2.23) + \frac{1}{2}(-0.94)(14.59),
\\
&= -22.73085 < -22.73,
\intertext{while}
y^{(8)} &> \frac{35}{2}(2.12)(-1.13) + \frac{21}{2}(-0.15)(1.93)
+ \frac{7}{2}(0.29)(-3.38) + \frac{1}{2}(-1)(27.12)
\\
&= -61.95345 > -61.96.
\end{align*}
Therefore, we obtain the bounds
\begin{equation}
\label{YY8}
\boxed{-61.96 < y^{(8)}(x) < -22.73} \,.
\end{equation}

By \eqref{y9},
\begin{equation*}
y^{(9)} = \frac{35}{2} (y^{(4)})^2 + 28 y'''y^{(5)} + 14 y''y^{(6)}
+ 4 y'y^{(7)} + \frac{1}{2} yy^{(8)} \,.
\end{equation*}
Therefore, using \eqref{YY1} through \eqref{YY8}, we conclude that for
all $x \in I$,
\begin{align*}
y^{(9)} &< \frac{35}{2}(-1.13)^2 + 28(2.12)(1.93) + 14(0.3)(2.23)
+ 4(0.29)(27.12) + \frac{1}{2}(-1)(-61.96)
\\
&= 146.75575 < 146.76,
\intertext{while}
y^{(9)} &> \frac{35}{2}(-0.74)^2 + 28(1.87)(1.17) + 14(-0.15)(2.23)
+ 4(.22)(14.59) + \frac{1}{2}(-0.94)(-27.73)
\\
&= 92.0335 > 92.03.
\end{align*}
Therefore, we obtain the bounds
\begin{equation}
\label{YY9}
\boxed{92.03 < y^{(9)}(x) < 146.76} \,.
\end{equation}

By \eqref{y10},
\begin{equation*}
y^{(10)} = 63 y^{(4)}y^{(5)} + 42 y'''y^{(6)} + 18 y''y^{(7)}
+ \frac{9}{2} y'y^{(8)} + \frac{1}{2} yy^{(9)} \,.
\end{equation*}
Therefore, using \eqref{YY1} through \eqref{YY9}, we conclude that for
all $x \in I$,
\begin{align*}
y^{(10)} &< 63(-0.74)(1.17) + 42(2.12)(2.23) + 18(0.3)(27.12)
+ \frac{9}{2}(0.22)(-22.73) 
\\
&\qquad\qquad + \frac{1}{2}(-0.94)(92.03)
\\
&= 280.9922 < 281,
\intertext{while}
y^{(10)} &> 63(-1.13)(1.93) + 42(2.12)(-3.38) + 18(-0.15)(27.12)
+ \frac{9}{2}(0.29)(-61.96) 
\\
&\qquad\qquad + \frac{1}{2}(-1)(146.76)
\\
&= -665.8137 > -665.9 \,.
\end{align*}
Therefore, we obtain the bounds
\begin{equation}
\label{YY10}
\boxed{-665.9 < y^{(10)}(x) < 281} \,.
\end{equation}
The inequality \eqref{YY10} was the goal of this long and detailed
computation!
\end{proof}

Now we can state the accuracy of the partial sum.

\begin{thm}
For all $x$ in $0 \leq x \leq 0.2$, the partial sum of degree~$9$:
\begin{align}
y_9(x) &:= -1 + \frac{1}{4}\,x - \frac{1}{16}\,x^2
+ \frac{67}{192}\,x^3 - \frac{35}{768}\,x^4 + \frac{69}{5120}\,x^5 
- \frac{77}{15360}\,x^6 + \frac{5317}{1032192}\,x^7
\nonumber \\
&\qquad - \frac{2655}{1835008}\,x^8 + \frac{80641}{165150720}\,x^9
\label{soll}
\end{align}
approximates the true value, $y(x)$, of the solution series
\eqref{sol}, with an error that does not exceed $2$~units in the
eleventh decimal place.
\end{thm}

\begin{proof}
The estimate
\begin{equation*}
|R_9(x)| = \biggl| \frac{y^{(10)}(\Theta_9)}{10!}\,x^{10} \biggr|
\leq \frac{665.9}{10!}\,(0.2)^{10} = 1.878\ldots\cdot 10^{-11}
< 2\cdot 10^{-11}
\end{equation*}
completes the proof.
\end{proof}


\section{Conclusions}

Our integral inequality \eqref{IE} can be applied to wide classes of
differential equations.

For example, our method allows us to prove that in the interval
$0 \leq x \leq 0.4$ the polynomial
\begin{equation*}
\overline{y}(x) := 1 + \frac{x}{4} + \frac{3}{16}\,x^2
+ \frac{7}{192}\,x^3 + \frac{1}{96}\,x^4 + \frac{1}{200}\,x^5
\end{equation*}
approximates the true solution, $y(x)$, of the initial value problem
\begin{equation*}
4y' = x+y^2,  \qquad  y(0) := 1,
\end{equation*}
with an error that does not exceed $2$~units in the fifth decimal
place.

We chose the \textsc{Riccati}--\textsc{Bernoulli} equation because it
illustrates the process so perfectly and because a direct estimate of
the accuracy of the partial sum of the series solution is troublesome.

We did not investigate the accuracy for the \emph{negative} half of
the interval, i.e., for $-0.2 \leq x \leq 0$, which we leave as an
exercise for the reader. The only change occurs in the estimate of
$y''(x)$ since then $x$ can take negative values.

Finally, we observe that we did not exploit the \emph{sign} of the
error. In fact, our estimate \eqref{YY10} allows us to say that the
error we commit is between an error in \emph{defect} smaller than
$2$~units in the eleventh decimal place and an error in \emph{excess}
smaller than $8$~units in the twelfth decimal place. Therefore we can
\emph{centralize the error} by adding the term
\begin{equation*}
\frac{281 - 665.9}{2\cdot 10!}\,x^{10}
= - \frac{1283}{24192000}\,x^{10}
\end{equation*}
to our polynomial \eqref{soll} and obtain an approximating polynomial
whose maximum error is between $\pm 1.4$~units in the $11$th decimal
place.

\subsubsection*{Acknowledgment}
We thank Joseph C. V\'arilly for a helpful discussion. Support from
the Vicerrector\'{\i}a de Investigaci\'on of the University of Costa
Rica is acknowledged.


\end{document}